\documentclass[11pt,reqno]{amsart}
\usepackage{amsmath,amssymb,mathrsfs}
\usepackage{graphicx,cite}
\usepackage{subfigure}
\usepackage{epstopdf}
\usepackage{graphics} 
\usepackage{epsfig} 
\usepackage{tikz}
\usepackage{paralist}
\usepackage{booktabs}
\usepackage{enumerate}
\usepackage{enumitem}
\usepackage{mdwlist}

\newtheorem{theorem}{Theorem}[section]

\newtheorem{lemma}[theorem]{Lemma}

\numberwithin{equation}{section}
\begin{document}

\title[An inverse cavity scattering problem]{Uniqueness of an inverse cavity scattering problem for the biharmonic wave equation}

\author{Heping Dong}
\address{School of Mathematics, Jilin University, Changchun,  Jilin 130012,  China}
\email{dhp@jlu.edu.cn}

\author{Peijun Li}
\address{Department of Mathematics, Purdue University, West Lafayette, Indiana
	47907, USA}
\email{lipeijun@math.purdue.edu}

\thanks{The first author was supported in part by the NSFC grant 12171201. The second author was supported partially by the NSF grant DMS-2208256.}



\begin{abstract}
This paper addresses an inverse cavity scattering problem associated with the biharmonic wave equation in two dimensions. The objective is to determine the domain or shape of the cavity. The Green's representations are demonstrated for the solution to the boundary value problem, and the one-to-one correspondence is confirmed between the Helmholtz component of biharmonic waves and the resulting far-field patterns. Two mixed reciprocity relations are deduced, linking the scattered field generated by plane waves to the far-field pattern produced by various types of point sources. Furthermore, the symmetry relations are explored for the scattered fields generated by point sources. Finally, we present two uniqueness results for the inverse problem by utilizing both far-field patterns and phaseless near-field data.
\end{abstract}

\maketitle

\section{Introduction}\label{S1}

Scattering problems associated with biharmonic waves hold significant importance in the domain of thin plate elasticity, exhibiting practical applications across diverse fields. These applications range from achievements in experimental ultra-broadband elastic cloaking \cite{FGE-09, SWW-12}, the utilization of platonic crystals for channeling destructive wave energy \cite{EP-07, WUW2004}, to the development of the acoustic black hole technique for vibration control and noise reduction \cite{PGCS-20}. From a mathematical perspective, the motion involving the out-of-plane displacement of thin plates in these structures complies with the governing fourth-order biharmonic wave equation. Despite the well-established scattering theories for acoustic, elastic, and electromagnetic waves, various scattering problems concerning biharmonic waves remain unresolved. This paper specifically addresses the challenges related to uniqueness in the context of the inverse scattering of biharmonic waves by a cavity in an infinite thin plate, thereby contributing to the understanding of this complex phenomenon.

In mathematical contexts, the inverse cavity scattering problem in biharmonic waves shares similarities with inverse obstacle scattering problems formulated in many other wave models. Within these models, various uniqueness results have been established concerning both inverse obstacle scattering problems and their corresponding phaseless variations. In acoustics, \cite{KK1993} expanded upon and simplified Isakov's uniqueness theorem \cite{I1990} by using classical solutions. This was achieved by establishing a contradiction between the pointwise limits of singular solutions. Notably, the applicability of the proof extends to the case involving the Neumann boundary condition. Not only is the shape of the obstacle uniquely determined by the far-field pattern for an infinite number of incident plane waves, but also the boundary condition. For the proofs concerning the uniqueness of the inverse acoustic and electromagnetic obstacle problems, we direct attention to \cite[Theorems 5.6 and 7.1]{DR-shu2}. These results rely on the utilization of mixed reciprocity relations. We refer to \cite{DR-shu2} for a comprehensive overview for inverse acoustic and electromagnetic obstacle problems. In the context of elastic waves, \cite{HH1993} showed that the obstacle's unique determination can be achieved by analyzing the far-field pattern associated with all incident plane waves, combined with a specific range of polarizations at a fixed frequency.

The challenge in the phaseless inverse scattering problem arises from the inherent translation invariance of the phaseless far-field pattern generated by incident plane waves. To handle this obstacle, the introduction of superpositions involving two plane waves as incident fields was proposed in \cite{ZZ2017}, effectively breaking the translation invariance. Further advancements, outlined in \cite{ZG18}, utilized combinations of a plane wave and a point source alongside a reference ball, enabling the unique determination of both the obstacle and the boundary condition from phaseless far-field data. Subsequently, \cite{XZZ2018} extended this approach, employing superpositions of two plane waves and a reference ball, eliminating prior assumptions about the unknown obstacle. This uniqueness result was extended in \cite{XZZ2020_1} to the inverse electromagnetic obstacle scattering problem. Recent developments, as highlighted in \cite{XZZ2020, ZGSL2020}, have established uniqueness results based on phaseless near-field data, overcoming the limitation of translation invariance observed in previous approaches. For the phaseless inverse elastic obstacle scattering, in \cite{JL2019}, the uniqueness result was shown with phaseless far-field pattern and a reference ball. Inspired by the work in \cite{ZGSL2020}, \cite{CD2023} formulated two uniqueness results, one from phaseless far-field data and the other from phaseless near-field data, respectively, without the necessity of a reference ball. For the uniqueness result concerning the phaseless inverse acoustic-elastic interaction problem, we refer to \cite{DLL-IP20}.

This work investigates the inverse scattering of biharmonic waves by a cavity in an infinite thin plate, where the plate wave dynamics satisfy the governing two-dimensional biharmonic wave equation. This model also aligns with the classical Kirchhoff--Love model, particularly in cases focusing solely on bending behavior. In \cite{BR2020}, the uniqueness was established for identifying impenetrable obstacles from multistatic near-field data and an analysis of the linear sampling method was also provided. Unlike the well-studied inverse obstacle scattering problems in acoustic, elastic, and electromagnetic waves, relatively less research has been conducted on the inverse cavity scattering problem for the biharmonic wave equation and its corresponding phaseless variant. A primary distinction lies in the inability to establish a one-to-one correspondence between the scattered field and its far-field pattern, a feature notably present in the acoustic scenario. Moreover, the complexity inherent in the boundary conditions hinder direct application of unique continuation methods.

Drawing from insights presented in \cite{DL2023}, which introduced a boundary integral formulation for the direct cavity scattering problem, we investigate the uniqueness aspect of the inverse cavity scattering problem associated with the biharmonic wave equation. This study incorporates both far-field patterns and phaseless near-field data. By introducing two auxiliary functions, we transform the scattering problem into a coupled boundary value problem, containing the Helmholtz and modified Helmholtz equations. To address the challenges posed by the biharmonic wave equation, following \cite{DR-shu2}, we deduce the Green's representations for both the reduced boundary value problem and the original boundary value problem. Using this approach, we develop two types of mixed reciprocity relations, linking the scattered field generated by plane waves to the far-field pattern produced by various types of point sources. Additionally, we explore symmetry relations associated with the scattered field from different point sources and examine the one-to-one correspondence between the Helmholtz component of the biharmonic waves and the resulting far-field patterns. Subsequently, we establish a uniqueness result derived from the far-field pattern with a fixed wavenumber. Building upon this outcome, we further obtain uniqueness based on phaseless near-field data.

In summary, this study concentrates on the inverse scattering of biharmonic waves by a cavity.  It contains two primary contributions: 

\begin{enumerate}

	\item Derivation of the Green's representations for the biharmonic wave scattering problem, exploration of symmetry relations in the scattered field, and formulation of two distinct types of mixed reciprocity relations.
	
	\item Establishment of two uniqueness results for the inverse problem derived from both far-field patterns and phaseless near-field data at a fixed wavenumber.
	
\end{enumerate}

The paper is structured as follows. Section 2 introduces the biharmonic wave equation and the problem formulation.
In Section 3, we present the Green's representations applied to the associated boundary value problems. Section 4 focuses on establishing uniqueness results through the utilization of mixed reciprocity relations and symmetry relations. Finally, Section 5 provides concluding remarks and outlines avenues for future research.

\section{Problem formulation}\label{S2}

Consider a cavity represented by the domain $D$ with a smooth boundary $\Gamma$ in a two-dimensional infinite elastic thin plate. The cavity receives illumination from a time-harmonic plane wave described by
\[
u^i(x)=\mathrm{e}^{\mathrm{i}\kappa d\cdot x},\quad x\in\mathbb R^2,
\]
where $d=(\cos\theta, \sin\theta)^\top$ is the unit vector in the incident direction, with $\theta\in [0,2\pi)$ denoting the incident angle, and $\kappa>0$ represents the wavenumber. 

The total field, denoted as $u$, satisfies the governing two-dimensional biharmonic wave equation
\begin{align}\label{uwe}
\Delta^2u-\kappa^4 u=0 \quad {\rm in} ~ \mathbb{R}^2\setminus \overline{D}. 
\end{align}
The cavity is assumed to comply with the clamped boundary conditions
\begin{align}\label{bc}
u=0,\quad \partial_{n} u=0 \quad {\rm on} ~ \Gamma,
\end{align}
where $n$ denotes the unit outward normal vector along $\Gamma$. 

The total field $u$ can be written as 
\[
u= u^i+ u^s,
\]
where $u^s$ represents the scattered field. Examining \eqref{uwe} and \eqref{bc}, it can be verified that the scattered field $u^s$ also satisfies the two-dimensional biharmonic wave equation
\begin{align}\label{vwe}
	\Delta^2u^s-\kappa^4 u^s=0 \quad {\rm in} ~ \mathbb{R}^2\setminus \overline{D},
\end{align}
subject to the nonhomogeneous Dirichlet boundary conditions
\begin{align}\label{bcv}
	u^s=-u^i, \quad \partial_n u^s=- \partial_n u^i \quad {\rm on} ~ \Gamma.
\end{align}
Furthermore, the scattered field $u^s$ satisfies the Sommerfeld radiation conditions
\begin{align}\label{sc}
	\lim_{r\to\infty}r^{\frac{1}{2}}(\partial_r
u^s-\mathrm{i}\kappa u^s)=0, \quad
	\lim_{r\to\infty}r^{\frac{1}{2}}(\partial_r
	\Delta u^s-\mathrm{i}\kappa \Delta u^s)=0, \quad r=|x|.
\end{align}

Following \cite{DL2023}, we consider two auxiliary functions
\begin{align}\label{vhvm}
	u^s_{\rm H} = -\frac{1}{2\kappa^2}(\Delta u^s - \kappa^2 u^s), \quad u^s_{\rm M} = \frac{1}{2\kappa^2}(\Delta u^s + \kappa^2 u^s).
\end{align}
It is clear to note that 
\begin{align}\label{decomposition}
	u^s=u^s_{\rm H}+u^s_{\rm M}, \quad \Delta u^s=\kappa^2(u^s_{\rm M}-u^s_{\rm H}).
\end{align}
By \eqref{vwe}, we have
\begin{align*}
	(\Delta - \kappa^2)(\Delta + \kappa^2)u^s=0 \quad {\rm in} ~ \mathbb{R}^2\setminus \overline{D},
\end{align*}
which is fulfilled if $u^s_{\rm H}$ satisfies the Helmholtz equation and $u^s_{\rm M}$ complies with the modified Helmholtz equation, respectively, i.e., 
\begin{align} \label{vHvM}
	\Delta u^s_{\rm H} + \kappa^2 u^s_{\rm H} = 0, \quad \Delta u^s_{\rm M} - \kappa^2 u^s_{\rm M} = 0\quad \text{in}~\mathbb{R}^2\setminus \overline{D}.
\end{align}
Moreover, $u^s_{\rm H}$ and $u^s_{\rm M}$ satisfy the coupled boundary conditions on $\Gamma$: 
\begin{align}\label{bccouple}
	u^s_{\rm H}+u^s_{\rm M}= f_1,\quad \partial_n u^s_{\rm H} + \partial_n u^s_{\rm M}= f_2, 
\end{align}
where $f_1=-u^i$, $f_2=-\partial_n u^i$. It follows from \eqref{sc} and \eqref{vhvm} that $u^s_{\rm H}$ and $u^s_{\rm M}$ satisfy the Sommerfeld radiation condition 
\begin{align}\label{sc1}
	\lim_{r\to\infty}r^{\frac{1}{2}}(\partial_r
	u^s_{\rm H}-\mathrm{i}\kappa u^s_{\rm H})=0, \quad \lim_{r\to\infty}r^{\frac{1}{2}}(\partial_r
	u^s_{\rm M}-\mathrm{i}\kappa u^s_{\rm M})=0, \quad r=|x|.
\end{align}

By employing \eqref{vhvm} and \eqref{decomposition}, it becomes evident that the scattering problem described in \eqref{vwe}--\eqref{sc} is equivalent to the scattering problem defined in \eqref{vHvM}--\eqref{sc1}.

It is known that a radiating solution of \eqref{vwe} exhibits the following asymptotic expansion:
\begin{equation}\label{farf}
	u^s(x)=\frac{\mathrm{e}^{\mathrm{i}\kappa|x|}}{|x|^{\frac{1}{2}}}u^\infty(\hat{x})+\mathcal{O}\left(\frac{1}{|x|^{\frac{3}{2}}}\right),
	\quad |x|\to\infty
\end{equation}
consistently observed across all directions $\hat{x}:=x/|x|$. Here, $u^\infty$ denotes a function defined over the unit circle $\Omega$ and characterizes the far-field pattern exhibited by $u^s$. The inverse problem involves determining the cavity $D$ based on either the wave field $u^s$ or its corresponding far-field pattern $u^\infty$.

\section{Green's representation formulas}\label{S3}

This section concerns several Green's representation formulas related to the solution of the biharmonic wave equation.

Let $G_{\rm H}$ and $G_{\rm M}$ be the fundamental solutions corresponding to the Helmholtz equation and the modified Helmholtz equation in two dimensions, respectively. Explicitly, we have 
\begin{align}\label{GHGM}
	G_{\rm H}(x,y)=\frac{\mathrm{i}}{4}H_0^{(1)}(\kappa|x-y|),\quad G_{\rm M}(x,y)=\frac{\mathrm{i}}{4}H_0^{(1)}(\mathrm{i}\kappa|x-y|),\quad x\neq y, 
\end{align}
where $H_0^{(1)}$ denotes the Hankel function of the first kind with order zero. The Green's function of the biharmonic wave equation \eqref{uwe} can be verified to have a representation provided by
\begin{align}\label{funsol}
G(x,y)=\frac{1}{2\kappa^2}\big(G_{\rm M}(x,y)-G_{\rm H}(x,y)\big), \quad x\neq y.
\end{align}

According to Green's second theorem, for $u\in\mathcal{C}^4(\overline{D})$ and $v\in\mathcal{C}^2(\overline{D})$, we have
\begin{equation} \label{Green1}
	\int_{D} \big\{(\Delta^2 u) v - \Delta u\Delta v\big\}\mathrm{d}x = \int_{\Gamma} \bigg\{v\frac{\partial \Delta u}{\partial n}-\Delta u\frac{\partial v}{\partial{n}}\bigg\} \mathrm{d}s.
\end{equation}
Interchanging the roles of $u$ and $v$ and subtracting the resulting equation from \eqref{Green1} leads to the following identity for $u, v\in\mathcal{C}^4(\overline{D})$:
\begin{equation} \label{Green2}
	\int_{D} \big\{(\Delta^2 u) v - (\Delta^2 v) u\big\}\mathrm{d}x = \int_{\Gamma} \bigg\{v\frac{\partial \Delta u}{\partial{n}}-\Delta u\frac{\partial v}{\partial{n}}+\Delta v\frac{\partial u}{\partial{n}}-u\frac{\partial \Delta v}{\partial{n}}\bigg\} \mathrm{d}s.
\end{equation}

Let $Pv:=\Delta v$, $Qv:=-\partial_{n}\Delta v$, and define
	\begin{align*}
	W(v,\partial_{n}v)&=\int_{\Gamma}\biggl\{\big(P_yG(x,y)\big)\frac{\partial v}{\partial n}(y)+\big(Q_yG(x,y)\big)v(y)\biggr\}\mathrm{d}s(y),\\
	U(Pv,Qv)&=\int_{\Gamma}\biggl\{G(x,y)(Qv)(y)+\frac{\partial G(x,y)}
	{\partial n(y)}(Pv)(y)\biggr\}\mathrm{d}s(y).
	\end{align*}

The following result relates $u^s$ in the domain $D$ to its boundary values, where $W(u^s,\partial_{n}u^s)$ and $U(Pu^s,Qu^s)$ capture specific integrals over the boundary $\Gamma$ involving certain derivatives of $u^s$.    
    
\begin{theorem} \label{Greenth1}
Consider $u^s$ as a solution to the biharmonic wave equation \eqref{vwe} in the domain $D$ with a boundary $\Gamma\in \mathcal{C}^2$. Let $u^s\in \mathcal{C}^4(D)\cap \mathcal{C}(\overline{D})$, $\Delta u^s\in \mathcal{C}(\overline{D})$, and both $u^s$ and $\Delta u^s$ have normal derivatives on the boundary such that the limits
\begin{align*}
\frac{\partial u^s}{\partial n}(x)&=\lim_{h\to +0}n(x)\cdot \nabla u^s(x-h n(x)), \quad x\in\Gamma, \\
\frac{\partial \Delta u^s}{\partial n}(x)&=\lim_{h\to +0}n(x)\cdot \nabla \Delta u^s(x-hn(x)), \quad x\in\Gamma
\end{align*}
uniformly exist on $\Gamma$. Consequently,
\begin{align}\label{Bihar1}
u^s(x)=W(u^s,\partial_{n}u^s)-U(Pu^s,Qu^s), \quad x\in D.
\end{align}
\end{theorem}

\begin{proof}
First, we assume that $u^s\in \mathcal{C}^4(\overline{D})$. For any $x\in D$, let $\partial{B(x;\rho)}:=\{y\in \mathbb{R}^2: |x-y|=\rho\}$ denote a circle centered at $x$ with a radius $\rho$, where the unit normal $n$ is oriented towards the interior of $\partial{B(x;\rho)}$. Utilizing \eqref{Green2} for the functions $u^s$ and $G(x,\cdot)$ in the domain $D_\rho:=\{y\in D: |x-y|>\rho\}$, we derive from the Green's second theorem that 
\begin{align*}
	0&=\int_{D_\rho} \big\{G(x,y)\Delta^2 u^s(y) - u^s(y)\Delta_y^2 G(x,y) \big\}\mathrm{d}x \\
	&=\int_{\Gamma\cup \partial{B(x;\rho)}}\bigg\{G(x,y)\frac{\partial \Delta u^s}{\partial n}(y)-\Delta u^s\frac{\partial G(x,y)}{\partial n(y)}\\
	&\quad +\Delta_y G(x,y)\frac{\partial u^s}{\partial n}(y) - u^s(y)\frac{\partial\Delta G(x,y)}{\partial n(y)}\bigg\}\mathrm{d}s(y),
\end{align*}
which can be equivalently written as
\begin{align*}
	&U(Pu^s,Qu^s)-W(u^s,\partial_{n}u^s)\\
	&=\int_{\partial{B(x;\rho)}}\biggl\{G(x,y)(Qu^s)(y)+\frac{\partial G(x,y)}
	{\partial n(y)}(Pu^s)(y) \\
	&\quad - \big(P_yG(x,y)\big)\frac{\partial u^s}{\partial n}(y)-\big(Q_yG(x,y)\big)u^s(y)\biggr\}\mathrm{d}s(y).
\end{align*}

On $\partial{B(x;\rho)}$, the following equations hold:
\begin{align*}
	\nabla_yG(x,y)&=\frac{\mathrm{i}}{8\kappa}\frac{y-x}{\rho}\big[H_1^{(1)}(\kappa\rho)-\mathrm{i}H_1^{(1)}(\mathrm{i}\kappa\rho)\big],\\
	P_yG(x,y)&=\frac{\mathrm{i}}{\kappa}\big[H_0^{(1)}(\kappa\rho)+H_0^{(1)}(\mathrm{i}\kappa\rho)\big],\\
	Q_yG(x,y)&=\frac{\mathrm{i}\kappa}{8}n(y)\cdot\frac{y-x}{\rho}\big[H_1^{(1)}(\kappa\rho)+\mathrm{i}H_1^{(1)}(\mathrm{i}\kappa\rho)\big].
\end{align*}
Performing a direct computation using the limiting forms of Bessel functions as described in \cite[\S 10.7]{NIST2010} and the power series \cite[$(10.8.2)$]{NIST2010} demonstrates that
\begin{align*}
&\lim_{\rho\to 0}\int_{\partial{B(x;\rho)}}\biggl\{G(x,y)(Qu^s)(y)+\frac{\partial G(x,y)}
{\partial n(y)}(Pu^s)(y) \\
&\quad - \big(P_yG(x,y)\big)\frac{\partial u^s}{\partial n}(y)-\big(Q_yG(x,y)\big)u^s(y)\biggr\}\mathrm{d}s(y)\\
&= 0+0-0-u^s(x),
\end{align*}
which implies that \eqref{Bihar1} holds.

When considering $u^s\in\mathcal{C}^4(D)\cap \mathcal{C}(\overline{D})$ and $\Delta u^s\in\mathcal{C}(\overline{D})$, along with the uniform convergence of the normal derivatives of $u^s$ and $\Delta u^s$, we can initially employ integrals over parallel surfaces to the boundary $\Gamma$ and subsequently approach the limit toward $\Gamma$.
\end{proof}

The following result establishes a relationship between the solution $u^s$ of the exterior problem for the biharmonic wave equation and specific boundary integral terms, where $U(Pu^s,Qu^s)$ and $W(u^s,\partial_{n}u^s)$ denote particular integral expressions over $\Gamma$ involving derivatives of $u^s$.

\begin{theorem} \label{Greenth2}
Consider $u^s$ as a radiating solution to the biharmonic wave equation \eqref{vwe} with boundary $\Gamma\in \mathcal{C}^2$. Let $u^s\in \mathcal{C}^4(\mathbb{R}^2\setminus \overline{D})\cap \mathcal{C}(\mathbb{R}^2\setminus {D})$, $\Delta u^s\in \mathcal{C}(\mathbb{R}^2\setminus {D})$, and both $u^s$ and $\Delta u^s$ have normal derivatives on the boundary such that the limits
\begin{align*}
\frac{\partial u^s}{\partial n}(x)&=\lim_{h\to +0}n(x)\cdot \nabla u^s(x+hn(x)), \quad x\in\Gamma, \\
\frac{\partial \Delta u^s}{\partial n}(x)&=\lim_{h\to +0}n(x)\cdot \nabla \Delta u^s(x+hn(x)), \quad x\in\Gamma
\end{align*}
uniformly exist on $\Gamma$. Consequently,
\begin{align}\label{Bihar2}
u^s(x)=U(Pu^s,Qu^s)-W(u^s,\partial_{n}u^s), \quad x\in \mathbb{R}^2\setminus \overline{D}.
\end{align}
\end{theorem}

\begin{proof} 
Let us assume, without loss of generality, that the original point $O\in D$, and consider a sufficiently large radius $r$ such that $D$ is enclosed within $B(O;r)$ with boundary $\partial B(O;r)$. Applying Theorem \ref{Greenth1} in $D_r:=\{y\in\mathbb{R}^2\setminus \overline{D}: |y|<r\}$ yields 
\begin{align}\label{vdecom}
	u^s(x) = I_{\partial B(O;r)}-I_{\Gamma}, \quad x\in D_r,
\end{align}
where
\begin{align*}
I_{\sigma}&=\int_{\sigma}\biggl\{\big(P_yG(x,y)\big)\frac{\partial u^s}{\partial n}(y)+\big(Q_yG(x,y)\big)u^s(y)\\
&\quad - G(x,y)(Qu^s)(y)-\frac{\partial G(x,y)} {\partial n(y)}(Pu^s)(y)\biggr\}\mathrm{d}s(y)
\end{align*}
for $\sigma=\Gamma$ or $\partial B(O;r)$. 

We proceed by decomposing $I_{\partial B(O;r)}$ as follows:
\begin{align*}
I_{\partial B(O;r)}=(I_3-I_4)-(I_1-I_2),
\end{align*}
where
\begin{align*}
	I_1&:=\int_{\partial B(O;r)}\Delta u^s(y)\biggl\{\frac{\partial G(x,y)}{\partial n(y)}-\mathrm{i}\kappa G(x,y)\biggr\}\mathrm{d}s(y),\\
	I_2&:=\int_{\partial B(O;r)}G(x,y)\biggl\{\frac{\partial \Delta u^s}{\partial n}(y)-\mathrm{i}\kappa\Delta u^s(y)\biggr\}\mathrm{d}s(y),\\
	I_3&:=\int_{\partial B(O;r)}\Delta_y G(x,y)\biggl\{\frac{\partial u^s}{\partial n}(y)-\mathrm{i}\kappa u^s(y)\biggr\}\mathrm{d}s(y),\\
	I_4&:=\int_{\partial B(O;r)}u^s(y)\biggl\{\frac{\partial \Delta G(x,y)}{\partial n(y)}-\mathrm{i}\kappa\Delta_y G(x,y)\biggr\}\mathrm{d}s(y).	
\end{align*}
Noting \eqref{decomposition} and \cite[$(2.10)$]{DR-shu2}, we may show that 
\begin{align*}
	\int_{\partial B(O;r)} |u^s|^2 \mathrm{d}s=\mathcal{O}(1),
	\quad \int_{\partial B(O;r)} |\Delta u^s|^2 \mathrm{d}s=\mathcal{O}(1), \quad r\to \infty.
\end{align*}
In view of the radiation condition \eqref{sc} and the Cauchy--Schwarz inequality, we obtain 
\begin{equation*}
	I_j\to0, \quad r\to\infty, \quad j=1, 2, 3, 4.
\end{equation*}
The proof is completed by taking the limit as $r\to\infty$ in \eqref{vdecom} and recognizing that $I_{\Gamma}=W(u^s,\partial_{n}u^s)-U(Pu^s,Qu^s)$.
\end{proof}	

The equations \eqref{Bihar1} and \eqref{Bihar2} represent the Green's representations for the solution of the biharmonic wave equation. Using the definition of the far-field pattern in \eqref{farf}, the fundamental solutions in \eqref{GHGM}--\eqref{funsol}, and considering the exponential decay of $G_{\rm M}$, we derive that the far-field pattern of the scattered field for the biharmonic wave equation is given by 
\begin{align}\label{Bi_farfield}
u^\infty(\hat{x})&=\frac{1}{2}\frac{e^{\mathrm{i}\pi/4}}{\sqrt{8\kappa\pi}}\int_{\Gamma}\bigg\{u^s(y)\frac{\partial \mathrm{e}^{-\mathrm{i}\kappa\hat{x}\cdot y}}{\partial n(y)}-\mathrm{e}^{-\mathrm{i}\kappa\hat{x}\cdot y}\frac{\partial u^s}{\partial n}(y)\bigg\}\mathrm{d}s(y)\nonumber\\
&\quad -\frac{1}{2\kappa^2}\frac{e^{\mathrm{i}\pi/4}}{\sqrt{8\kappa\pi}}\int_{\Gamma}\bigg\{\Delta u^s(y)\frac{\partial \mathrm{e}^{-\mathrm{i}\kappa\hat{x}\cdot y}}{\partial n(y)}-\mathrm{e}^{-\mathrm{i}\kappa\hat{x}\cdot y}\frac{\partial \Delta u^s}{\partial n}(y)\bigg\}\mathrm{d}s(y), \quad \hat{x}\in\Omega.
\end{align}

Let $u^s_{\rm H}$ and $u^s_{\rm M}$ denote the radiating solutions to the problem \eqref{vHvM}--\eqref{sc1}. Similar to the proof presented in \cite[Theorem 2.5]{DR-shu2}, we can demonstrate that for $x\in \mathbb{R}^2\setminus \overline{D}$, the following integral equations hold:
\begin{align}
	\label{v_H}
		u^s_{\rm H}(x)&=\int_{\Gamma}\bigg\{u^s_{\rm H}(y)\frac{\partial G_{\rm H}(x,y)}{\partial n(y)}-G_{\rm H}(x,y)\frac{\partial u^s_{\rm H}}{\partial n}(y)\bigg\}\mathrm{d}s(y),\\
		 \label{v_M}
		u^s_{\rm M}(x)&=\int_{\Gamma}\bigg\{u^s_{\rm M}(y)\frac{\partial G_{\rm M}(x,y)}{\partial n(y)}-G_{\rm M}(x,y)\frac{\partial u^s_{\rm M}}{\partial n}(y)\bigg\}\mathrm{d}s(y).
	\end{align}
Using the asymptotic behavior of the fundamental solution $G_{\rm H}$ to the Helmholtz equation, we derive 
	\begin{align}\label{far}
		u_{\rm H}^\infty(\hat{x})=\frac{e^{\mathrm{i}\pi/4}}{\sqrt{8\kappa\pi}}\int_{\Gamma}\bigg\{u^s_{\rm H}(y)\frac{\partial \mathrm{e}^{-\mathrm{i}
				\kappa\hat{x}\cdot y}}{\partial n(y)}-\mathrm{e}^{-\mathrm{i}
			\kappa\hat{x}\cdot y}\frac{\partial u^s_{\rm H}}{\partial n}(y)\bigg\}\mathrm{d}s(y), \quad \hat{x}\in\Omega.
	\end{align}
	
Given the decomposition \eqref{decomposition} and the fundamental solutions \eqref{GHGM}--\eqref{funsol}, the combination of $u^s_{\rm H}$ and $u^s_{\rm M}$, as expressed by \eqref{v_H}--\eqref{v_M}, equals $U(Pu^s,Qu^s)-W(u^s,\partial_n u^s)$, which can be represented as $u^s_{\rm H}+u^s_{\rm M}=U(Pu^s,Qu^s)-W(u^s,\partial_n u^s)$. Additionally, it is evident that $u^\infty(\hat{x})=u_{\rm H}^\infty(\hat{x})$.

Furthermore, as mentioned in \cite[Remark 2.1]{DL2023}, it has been established that $u^s_{\rm M}$ and $\partial_r u^s_{\rm M}$ exhibit exponential decay as $|x|\to\infty$ for the fixed wavenumber $\kappa$ or as $\kappa|x|\to\infty$. Specifically, $u^s_{\rm M}$ decays according to the expression
\begin{align*}
u^s_{\rm M}(x)=\mathcal{O}\bigg(\frac{\mathrm{e}^{-\kappa |x|}}{|x|^{\frac{1}{2}}}\bigg), \quad |x|\to\infty.
\end{align*}

\section{Uniqueness}\label{S4}

This section is dedicated to establishing the uniqueness of the inverse cavity scattering problem derived from both far-field patterns and phaseless near-field data. The subsequent two lemmas address the one-to-one correspondence between the Helmholtz component of biharmonic waves and their far-field patterns.

\begin{lemma}\label{Rellich}
Consider $u^s\in \mathcal{C}^4(\mathbb{R}^2\setminus \overline{D})$ as a solution to the biharmonic wave equation \eqref{vwe}, satisfying
\begin{align}\label{Rellicheqn}
\lim_{r\to\infty}\int_{|x|=r}|u^s(x)|^2\mathrm{d}s=0. 
\end{align}
Then, it follows that $u^s_{\rm H}=0$ in $\mathbb{R}^2\setminus \overline{D}$. 
\end{lemma}

\begin{proof}
Given that $u^s_{\rm M}$ decays exponentially and $u^s_{\rm H}$ is bounded in $\mathbb{R}^2\setminus\overline{D}$, we have
\begin{align*}
\lim_{r\to\infty}\int_{|x|=r}|u^s_{\rm H}(x)|^2\mathrm{d}s=\lim_{r\to\infty}\int_{|x|=r}\Big(|u^s(x)|^2+|u^s_{\rm M}(x)|^2-2\Re(u^s(x)\overline{u^s_{\rm M}(x)})\Big)\mathrm{d}s=0, 
\end{align*}
which completes the proof by invoking Rellich's lemma \cite[Lemma 2.12]{DR-shu2}.
\end{proof}

\begin{lemma} \label{Rellich1}
Consider a solution $u^s\in \mathcal{C}^4(\mathbb{R}^2\setminus \overline{D})$ to the biharmonic wave equation, wherein the far-field pattern $u^\infty=0$. Then, it follows that $u^s_{\rm H}=0$ in $\mathbb{R}^2\setminus \overline{D}$. 
\end{lemma}

\begin{proof}
From \eqref{farf}, we have
\begin{align*}
\int_{|x|=r}|u^s(x)|^2\,\mathrm{d}s=\int_{\Omega}|u^\infty(\hat{x})|^2\mathrm{d}s+\mathcal{O}\Big(\frac{1}{r}\Big), \quad r\to\infty.
\end{align*}
The condition $u^\infty=0$ indicates that \eqref{Rellicheqn} is fulfilled. Therefore, the lemma is an immediate consequence of Lemma \ref{Rellich}.
\end{proof} 

For point sources $w^i(x,z)=G_{\rm H}(x,z)$ and $v^i(x,z)=G(x,z)$ located at $z\in\mathbb{R}^2\setminus \overline{D}$, we denote the corresponding total fields as $w(x,z)$ and $v(x,z)$, respectively. These fields can be decomposed into $w=w^i+w^s$ and $v=v^i+v^s$. Here, the scattered fields are represented by $w^s=w^s_{\rm H}+w^s_{\rm M}$ and $v^s=v^s_{\rm H}+v^s_{\rm M}$, while the far-field patterns corresponding to $w^s_{\rm H}$ and $v^s_{\rm H}$ are denoted as $w_{\rm H}^\infty(\hat{x},z)$ and $v_{\rm H}^\infty(\hat{x},z)$, respectively. The subsequent results focus on the mixed reciprocity and symmetry relations concerning these far-field patterns and scattered fields.

\begin{theorem}\label{Mixed}
For $z\in\mathbb{R}^2\setminus \overline{D}$ and $d\in\Omega$, the following relations hold: 
\begin{align}\label{mixrel}
\frac{\sqrt{8\kappa\pi}}{e^{\mathrm{i}\pi/4}} w_{\rm H}^\infty(-d,z)=u^s_{\rm H}(z,d),\quad \frac{\sqrt{8\kappa\pi}}{e^{\mathrm{i}\pi/4}} v_{\rm H}^\infty(-d,z)=-\frac{1}{2\kappa^2}u^s(z,d).
\end{align}
\end{theorem}

\begin{proof}
By using Green's theorem and the radiation condition for $u^s_{\rm H}$ and $w^s_{\rm H}$, we have 
\begin{align*}
\int_{\Gamma}\Big\{w^i(\cdot,z)\frac{\partial}{\partial n}u^i(\cdot,d)-u^i(\cdot,d)\frac{\partial}{\partial n}w^i(\cdot,z)\Big\}\mathrm{d}s &=0, \\
\int_{\Gamma}\Big\{w^s_{\rm H}(\cdot,z)\frac{\partial}{\partial n}u^s_{\rm H}(\cdot,d)-u^s_{\rm H}(\cdot,d)\frac{\partial}{\partial n}w^s_{\rm H}(\cdot,z)\Big\}\mathrm{d}s &=0.
\end{align*}	
It follows from \eqref{v_H} and \eqref{far} that
\begin{align*}
\int_{\Gamma}\Big\{w^s_{\rm H}(\cdot,z)\frac{\partial}{\partial n}u^i(\cdot,d)-u^i(\cdot,d)\frac{\partial}{\partial n}w^s_{\rm H}(\cdot,z)\Big\}\mathrm{d}s &=\frac{\sqrt{8\kappa\pi}}{e^{\mathrm{i}\pi/4}}w_{\rm H}^\infty(-d,z),\\
\int_{\Gamma}\Big\{u^s_{\rm H}(\cdot,d)\frac{\partial}{\partial n}w^i(\cdot,z)-w^i(\cdot,z)\frac{\partial}{\partial n}u^s_{\rm H}(\cdot,d)\Big\}\mathrm{d}s &=u^s_{\rm H}(z,d).
\end{align*}
Subtracting the last equation from the sum of the three preceding equations and incorporating the boundary condition \eqref{bccouple}, we obtain
\begin{align*}
&\frac{\sqrt{8\kappa\pi}}{e^{\mathrm{i}\pi/4}}w_{\rm H}^\infty(-d,z)-u^s_{\rm H}(z,d)\\
&=\int_{\Gamma}\Big\{\big(w^i(\cdot,z)+w^s_{\rm H}(\cdot,z)\big)\frac{\partial}{\partial n}\big(u^i(\cdot,d)+u^s_{\rm H}(\cdot,d)\big)\\
&\quad -\big(u^i(\cdot,d)+u^s_{\rm H}(\cdot,d)\big)\frac{\partial}{\partial n}\big(w^i(\cdot,z)+w^s_{\rm H}(\cdot,z)\big)\Big\}\mathrm{d}s\\
&=\int_{\Gamma}\Big\{w^s_{\rm M}(\cdot,z)\frac{\partial}{\partial n}u^s_{\rm M}(\cdot,d)-u^s_{\rm M}(\cdot,d)\frac{\partial}{\partial n}w^s_{\rm M}(\cdot,z)\Big\}\mathrm{d}s.
\end{align*}
The proof of the first equation in \eqref{mixrel} is completed through the application of Green's theorem in $D_R:= \{y\in\mathbb{R}^2\setminus \overline{D}: |y|<R\}$, along with the utilization of the Cauchy--Schwarz inequality, and considering the exponential decay of $u^s_{\rm M}$ and $w^s_{\rm M}$. 

Next, we prove the second equation in \eqref{mixrel}. By using Green's theorem \eqref{Green1} and the radiation condition for $u^s_{\rm H}$ and $w^s_{\rm H}$, we have 
\begin{align}
	&\int_{\Gamma}\Big\{v^i(\cdot,z)\frac{\partial}{\partial n}\Delta u^i(\cdot,d)-\Delta u^i(\cdot,d)\frac{\partial}{\partial n}v^i(\cdot,z)\nonumber\\
	&\quad +\Delta v^i(\cdot,z)\frac{\partial}{\partial n}u^i(\cdot,d)-u^i(\cdot,d)\frac{\partial}{\partial n}\Delta v^i(\cdot,z)\Big\}\mathrm{d}s=0, \label{MR1}\\
	-\frac{1}{2\kappa^2}&\int_{\Gamma}\Big\{v^s(\cdot,z)\frac{\partial}{\partial n}\Delta u^s(\cdot,d)-\Delta u^s(\cdot,d)\frac{\partial}{\partial n}v^s(\cdot,z)\nonumber\\
	&\quad +\Delta v^s(\cdot,z)\frac{\partial}{\partial n}u^s(\cdot,d)-u^s(\cdot,d)\frac{\partial}{\partial n}\Delta v^s(\cdot,z)\Big\}\mathrm{d}s=0. \label{MR2}	
\end{align}	
Noting $\Delta u^i + \kappa^2 u^i = 0 $ and the far-field pattern \eqref{Bi_farfield}, we obtain 
\begin{align} \label{MR3}
	&\frac{1}{2\kappa^2}\int_{\Gamma}\Big\{-v^s(\cdot,z)\frac{\partial}{\partial n}\Delta u^i(\cdot,d)+\Delta u^i(\cdot,d)\frac{\partial}{\partial n}v^s(\cdot,z)\Big\}\mathrm{d}s\nonumber\\
	&\quad -\frac{1}{2\kappa^2}\int_{\Gamma}\Big\{\Delta v^s(\cdot,z)\frac{\partial}{\partial n}u^i(\cdot,d)-u^i(\cdot,d)\frac{\partial}{\partial n}\Delta v^s(\cdot,z)\Big\}\mathrm{d}s=\frac{\sqrt{8\kappa\pi}}{e^{\mathrm{i}\pi/4}}v_{\rm H}^\infty(-d,z). 
\end{align}	
It follows from the Green's representation \eqref{Bihar2} that 
\begin{align} \label{MR4}
	&\int_{\Gamma}\Big\{-v^i(\cdot,z)\frac{\partial}{\partial n}\Delta u^s(\cdot,d)+\Delta u^s(\cdot,d)\frac{\partial}{\partial n}v^i(\cdot,z)\Big\}\mathrm{d}s\nonumber\\
	&\quad -\int_{\Gamma}\Big\{\Delta v^i(\cdot,z)\frac{\partial}{\partial n}u^s(\cdot,d)-u^s(\cdot,d)\frac{\partial}{\partial n}\Delta v^i(\cdot,z)\Big\}\mathrm{d}s=u^s(z,d).
\end{align}
By adding \eqref{MR2} and \eqref{MR3}, subtracting \eqref{MR1} from \eqref{MR4}, and using the boundary condition \eqref{bc}, we arrive at 
\begin{align*}
	&\frac{\sqrt{8\kappa\pi}}{e^{\mathrm{i}\pi/4}}v_{\rm H}^\infty(-d,z)+\frac{1}{2\kappa^2}u^s(z,d)\\
	&=-\frac{1}{2\kappa^2}\int_{\Gamma}\Big\{v(\cdot,z)\frac{\partial}{\partial n}\Delta u(\cdot,d)-\Delta u(\cdot,d)\frac{\partial}{\partial n}v(\cdot,z)\\
	&\quad +\Delta v(\cdot,z)\frac{\partial}{\partial n}u(\cdot,d)-u(\cdot,d)\frac{\partial}{\partial n}\Delta v(\cdot,z)\Big\}\mathrm{d}s=0,
\end{align*}
which completes the proof. 
\end{proof}

\begin{theorem}\label{Reci}
	For the scattering of point sources $w^i=G_\sigma$, where $\sigma=H, M$, the following relation holds:
	\begin{align}\label{symm}
		w^s_\sigma(x,z)=w^s_\sigma(z,x), \quad \forall\, x, z\in\mathbb{R}^2\setminus \overline{D}.
	\end{align}
\end{theorem}

\begin{proof}
Analogous to the proof of Theorem \ref{Mixed}, employing Green's theorem and considering the radiation condition for $w^s_{\rm H}$ and $w^s_{\rm M}$, we derive 
	\begin{align*}
		\int_{\Gamma}\Big\{w^i(\cdot,z)\frac{\partial}{\partial n}w^i(\cdot,x)-w^i(\cdot,x)\frac{\partial}{\partial n}w^i(\cdot,z)\Big\}\mathrm{d}s &=0, \\
		\int_{\Gamma}\Big\{w^s_\sigma(\cdot,z)\frac{\partial}{\partial n}w^s_\sigma(\cdot,x)-w^s_\sigma(\cdot,x)\frac{\partial}{\partial n}w^s_\sigma(\cdot,z)\Big\}\mathrm{d}s&=0,
	\end{align*}	
along with the solution representations
	\begin{align*}
		&\int_{\Gamma}\Big\{w^s_\sigma(\cdot,z)\frac{\partial}{\partial n}w^i(\cdot,x)-w^i(\cdot,x)\frac{\partial}{\partial n}w^s_\sigma(\cdot,z)\Big\}\mathrm{d}s=w^s_\sigma(x,z),\\
		&\int_{\Gamma}\Big\{w^s_\sigma(\cdot,x)\frac{\partial}{\partial n}w^i(\cdot,z)-w^i(\cdot,z)\frac{\partial}{\partial n}w^s_\sigma(\cdot,x)\Big\}\mathrm{d}s=w^s_\sigma(z,x).
	\end{align*}
Again, subtracting the last equation from the sum of the three preceding equations and using the boundary condition \eqref{bccouple}, we obtain 
	\begin{align*}
		&w^s_\sigma(x,z)-w^s_\sigma(z,x)\\
		&=\int_{\Gamma}\Big\{\big(w^i(\cdot,z)+w^s_\sigma(\cdot,z)\big)\frac{\partial}{\partial n}\big(w^i(\cdot,x)+w^s_\sigma(\cdot,x)\big)\\
		&\quad -\big(w^i(\cdot,x)+w^s_\sigma(\cdot,x)\big)\frac{\partial}{\partial n}\big(w^i(\cdot,z)+w^s_\sigma(\cdot,z)\big)\Big\}\mathrm{d}s\\
		&=\int_{\Gamma}\Big\{w^s_{\sigma'}(\cdot,z)\frac{\partial}{\partial n}w^s_{\sigma'}(\cdot,x)-w^s_{\sigma'}(\cdot,x)\frac{\partial}{\partial n}w^s_{\sigma'}(\cdot,z)\Big\}\mathrm{d}s=0,
	\end{align*}
where $\sigma'=M, H$, which implies that \eqref{symm} holds.
\end{proof}

Based on Theorem \ref{Reci} and the well-posedness of the exterior boundary value problem for  $w^s_{\sigma'}$, we have $w^s_{\sigma'}(x,z)=w^s_{\sigma'}(z,x)$. Using the superposition principle of the scattered field, we can derive the reciprocity relation for the incident wave $v^i(x,z)=G(x,z)$, expressed as
	\begin{align}\label{symmv}
	v^s(x,z)=v^s(z,x), \quad \forall\, x, z\in\mathbb{R}^2\setminus \overline{D}.
\end{align}

Now, we demonstrate the uniqueness obtained from far-field patterns for the inverse cavity scattering problem of the biharmonic wave equation. 

\begin{theorem}\label{uniqueness1}
	Let $D_1$ and $D_2$ be two cavities meeting the boundary condition \eqref{bc}, with corresponding far-field patterns
    $u_1^\infty$ and $u_2^\infty$ satisfying 
	\begin{align} \label{UniqueD1D2}
		u_{1}^\infty(\hat{x},d)=u_{2}^\infty(\hat{x},d),\quad\forall\, \hat{x}, d\in\Omega.
	\end{align}
Then $D_1=D_2$.
\end{theorem}

\begin{proof}
Building upon \eqref{UniqueD1D2} and Lemma \ref{Rellich}, we deduce that the respective scattered fields related to $D_1$ and $D_2$ satisfy 
	\begin{align*}
		u^s_{{\rm H},1}(x,d)=u^s_{{\rm H},2}(x,d), \quad x\in\mathbb{R}^2\setminus \overline{D_1\cup D_2}.
	\end{align*} 
It follows from the mixed reciprocity relation \eqref{mixrel} that 
\begin{align*}
	w_{{\rm H},1}^\infty(-d,x)=w_{{\rm H},2}^\infty(-d,x). 
\end{align*}
With the help of Lemma \ref{Rellich} and Theorem \ref{Reci}, and noting the continuity of the scattered field, we obtain
\begin{align}\label{wHequ}
w^s_{{\rm H},1}(x,z)=w^s_{{\rm H},2}(x,z), \quad\forall\, x, z\in\mathbb{R}^2\setminus \{D_1\cup D_2\}.
\end{align}	
Consequently, we have
\begin{align}\label{wMequ}
w^s_{{\rm M},1}(x,z)=w^s_{{\rm M},2}(x,z), \quad\forall\, x, z\in\mathbb{R}^2\setminus \{D_1\cup D_2\},
\end{align}
because $w^s_{{\rm M},j}$, where $j=1,2$, are the exponentially decaying solutions to the same exterior boundary value problem:
\begin{equation*}
\left\{
\begin{aligned}
	&\Delta w^s_{{\rm M},j} - \kappa^2 w^s_{{\rm M},j} = 0 &&\quad {\rm in} ~ \mathbb{R}^2\setminus \{D_1\cup D_2\},\\
	& w^s_{{\rm M},j} = -w^s_{{\rm H},j}-w^i_{\rm H} &&\quad {\rm on} ~ \partial D_1\cup \partial D_2.
\end{aligned}
\right.
\end{equation*}

Applying proof by contradiction, suppose $D_1\neq D_2$. Without loss of generality, let $x^*\in \partial D_1$ and $x^*\notin \overline{D}_2$. Define 
\begin{align*}
	z_m=x^*+\frac{1}{m} n(x^*)\in\mathbb{R}^2\setminus \{D_1\cup D_2\}, \quad m=1,2,\cdots
\end{align*}
for sufficiently large $m$. On one hand, utilizing the reciprocity relation \eqref{symm} and the well-posedness of the direct scattering problem, we derive 
\begin{align}\label{contradiction1}
	\lim_{m\to\infty} w^s_{{\rm H},2}(x^*,z_m)=w^s_{{\rm H},2}(x^*,x^*).
\end{align}
On the other hand, considering \eqref{wMequ}, the boundary condition $w=0$, and the continuity as well as boundedness of the scattered field $w^s_{{\rm M},2}(x^*,\cdot)$, we have
\begin{align}\label{contradiction2}
	\lim_{m\to\infty} w^s_{{\rm H},1}(x^*,z_m)&=\lim_{m\to\infty}\big(-w^s_{{\rm M},1}(x^*,z_m)-w^i_{\rm H}(x^*,z_m)\big)\nonumber\\
	&=\lim_{m\to\infty} \big(-w^s_{{\rm M},2}(x^*,z_m)-w^i_{\rm H}(x^*,z_m)\big)\nonumber\\
	&=-w^s_{{\rm M},2}(x^*,x^*)-\lim_{m\to\infty} G_{\rm H}(x^*,z_m)=\infty.
\end{align}
In view of \eqref{contradiction1} and \eqref{contradiction2}, we observe a contradiction with \eqref{wHequ}. Consequently, we establish that $D_1=D_2$. 
\end{proof}

Next, we investigate uniqueness by employing phaseless near-field data. To initiate, we explore a translational characteristic exhibited by the far-field pattern associated with a domain $D$.

\begin{theorem}\label{ti}
Considering $D_h:=\{x+h:x\in D\}$ with $h\in \mathbb{R}^2$, the far-field pattern corresponding to the incident plane wave $u^i(x,d)=\mathrm{e}^{\mathrm{i}\kappa d\cdot x}$ satisfies the following relation: 
\begin{align}\label{transinvar}
u^\infty(\hat{x};D_h)=\mathrm{e}^{\mathrm{i}\kappa(d-\hat{x})\cdot h}u^\infty(\hat{x};D),\quad \hat x\in\Omega.
\end{align}
\end{theorem}

\begin{proof}
Based on the boundary condition, for $x\in\partial D_h$, we derive that 
\begin{align*}
u^s(x;D_h)=-\mathrm{e}^{\mathrm{i}\kappa d\cdot x}=-\mathrm{e}^{\mathrm{i}\kappa d\cdot (x-h)}\mathrm{e}^{\mathrm{i}\kappa d\cdot h}=u^s(x-h;D)\mathrm{e}^{\mathrm{i}\kappa d\cdot h}.
\end{align*}
By the uniqueness of the solution to the direct scattering problem \cite{DL2023},  it follows that
\begin{align*}
u^s(x;D_h)=\mathrm{e}^{\mathrm{i}\kappa d\cdot h}u^s(x-h;D), \quad \forall\, x\in\mathbb{R}^2\setminus D_h.
\end{align*}
Consequently, we have
\begin{align*}
\Delta u^s(x;D_h)=\mathrm{e}^{\mathrm{i}\kappa d\cdot h}\Delta u^s(x-h;D), \quad \forall\, x\in\mathbb{R}^2\setminus D_h,
\end{align*}
and
\begin{align*}
\frac{\partial u^s}{\partial n}(x;D_h)&=\mathrm{e}^{\mathrm{i}\kappa d\cdot h}\frac{\partial u^s}{\partial n}(x-h;D),\quad x\in\partial D_h,\\
\frac{\partial \Delta u^s}{\partial n}(x;D_h)&=\mathrm{e}^{\mathrm{i}\kappa d\cdot h}\frac{\partial\Delta u^s}{\partial n}(x-h;D),\quad x\in\partial D_h.
\end{align*}
Combining the above equations yields 
\begin{align*}
\int_{\partial D_h}\frac{\partial u^s}{\partial n}(y;D_h)\mathrm{e}^{-\mathrm{i}\kappa\hat{x}\cdot y}\mathrm{d}s(y)&=\int_{\partial D_h}\mathrm{e}^{\mathrm{i}\kappa d\cdot h}\frac{\partial u^s}{\partial n}(y-h;D)\mathrm{e}^{-\mathrm{i}\kappa\hat{x}\cdot (y-h)}\mathrm{e}^{-\mathrm{i}\kappa\hat{x}\cdot h}\mathrm{d}s(y)\\
&=\mathrm{e}^{\mathrm{i}\kappa(d-\hat{x})\cdot h}\int_{\partial D}\frac{\partial u^s}{\partial n}(y;D)\mathrm{e}^{-\mathrm{i}\kappa\hat{x}\cdot y}\mathrm{d}s(y).
\end{align*}
Similarly, we obtain 
\begin{align*}
\int_{\partial D_h}\frac{\partial\Delta u^s}{\partial n}(y;D_h)\mathrm{e}^{-\mathrm{i}\kappa\hat{x}\cdot y}\mathrm{d}s(y)&=\mathrm{e}^{\mathrm{i}\kappa(d-\hat{x})\cdot h}\int_{\partial D}\frac{\partial\Delta u^s}{\partial n}(y;D)\mathrm{e}^{-\mathrm{i}\kappa\hat{x}\cdot y}\mathrm{d}s(y), \\
\int_{\partial D_h}\frac{\partial\mathrm{e}^{-\mathrm{i}\kappa\hat{x}\cdot y}}{\partial n(y)}u^s(y;D_h)\mathrm{d}s(y)&=\mathrm{e}^{\mathrm{i}\kappa(d-\hat{x})\cdot h}\int_{\partial D}\frac{\partial\mathrm{e}^{-\mathrm{i}\kappa\hat{x}\cdot y}}{\partial n(y)}u^s(y;D)\mathrm{d}s(y),\\
\int_{\partial D_h}\frac{\partial\mathrm{e}^{-\mathrm{i}\kappa\hat{x}\cdot y}}{\partial n(y)}\Delta u^s(y;D_h)\mathrm{d}s(y)&=\mathrm{e}^{\mathrm{i}\kappa(d-\hat{x})\cdot h}\int_{\partial D}\frac{\partial\mathrm{e}^{-\mathrm{i}\kappa\hat{x}\cdot y}}{\partial n(y)}\Delta u^s(y;D)\mathrm{d}s(y),
\end{align*}
which imply that \eqref{transinvar} holds by noting \eqref{Bi_farfield}.
\end{proof}

Theorem \ref{ti} demonstrates that the phaseless far-field pattern admits a translation invariance property when utilizing a plane wave as an incident field. Hence, it is not feasible to recover 
the cavity's location solely based on the magnitude of the far-field pattern. To address this challenge, we introduce point sources into the scattering system and explore the uniqueness of the inverse problem by utilizing phaseless near-field data.

\begin{figure}
\centering
\includegraphics[width=0.4\textwidth]{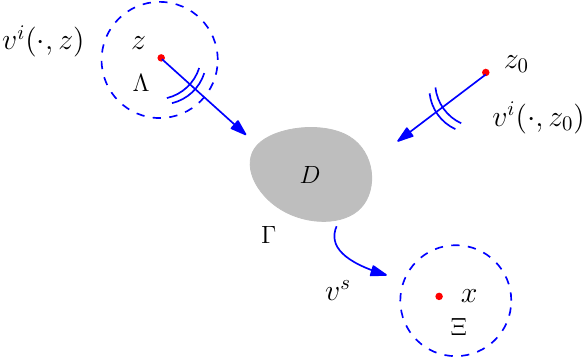}
\caption{A schematic for the configuration of the inverse cavity scattering problem utilizing phaseless near-field data.}
\label{pg}
\end{figure}

Consider the function $v^i(x, z_0)=G(x, z_0)$, where $z_0\in \mathbb R^2\setminus\overline D$ denotes the source location, representing a point source. Let $v^s(x, z_0)$ and $v(x, z_0)$ represent the corresponding scattered field and the total field, respectively. Similarly, define $v^i(x, z)=G(x, z)$, $v^s(x, z)$, and $v(x, z)$ as the incident field, scattered field, and total field, respectively, in relation to the point source located at $z\in\mathbb R^3\setminus \{\overline D\cup z_0\}$. Utilizing the principle of superposition, it can be deduced that $v(x, z_0)+v(x, z)$ constitutes the total field corresponding to the point incident field $G(x, z_0)+G(x, z)$. Figure \ref{pg} schematically illustrates the configuration of the problem.

Define the far-field pattern generated by a point source $v^i(x, z)$ as follows: 
\begin{align}\label{vinftytot}
v_{\rm tot}^\infty(\hat{x},z)=v_{\rm H}^\infty(\hat{x},z)+G^\infty(\hat{x},z),
\end{align}
where $v_{\rm H}^\infty$ is the far-field pattern corresponding to the scattered field $v^s(x, z)$, and $G^\infty(\hat{x},z)=-\frac{1}{2\kappa^2}\frac{\mathrm{e}^{\mathrm{i}\pi/4}}{\sqrt{8\kappa\pi}}\mathrm{e}^{-\mathrm{i}\kappa\hat{x}\cdot z}$ represents the far-field pattern produced by the point source of $G(x,z)$.

Building upon the motivation from \cite{ZGSL2020}, the subsequent result addresses the uniqueness aspect of the inverse cavity scattering problem, employing phaseless near-field data. 

\begin{theorem}\label{unithmphaseless}
Let $D_1$ and $D_2$ be two cavities with boundaries in $\mathcal{C}^3$. Suppose that $\Lambda$ and $\Xi$ are open domains such that $\Lambda\cap \Xi=\emptyset$, $\Lambda\cap D_j=\emptyset$, and $\Xi\cap D_j=\emptyset$, for $j=1,2$. Given a fixed wavenumber $\kappa$ and a fixed point $z_0\in\mathbb{R}^2\setminus\{\overline{D_1\cup D_2\cup\Lambda\cup\Xi}\}$, if the phaseless total fields $v_j$, where $j=1, 2$, satisfy the following conditions:
\begin{align}
|v_{1}(x,z_0)|&=|v_{2}(x,z_0)|, &&\forall\,{x}\in\Xi, \label{condi1}\\
|v_{1}(x,z)|&=|v_{2}(x,z)|, &&\forall\,(x,z)\in\Xi\times\Lambda, \label{condi2}\\
|v_{1}(x,z_0)+v_{1}(x,z)|&=|v_{2}(x,z_0)+v_{2}(x,z)|, &&\forall\,(x,z)\in\Xi\times\Lambda. \label{condi3}
\end{align}
Then $D_1=D_2$.
\end{theorem}

\begin{proof}
It follows from \eqref{condi1}--\eqref{condi3} that we have
\begin{align}\label{Reequal}
	\Re\{v_1(x,z_0)\overline{v_1(x,z)}\}=\Re\{v_2(x,z_0)\overline{v_2(x,z)}\}, \quad \forall\,(x,z)\in\Xi\times\Lambda.
\end{align}
Given that $v_j(x, z_0)$ and $v_j(x, z)$, where $j=1,2$, are complex-valued functions, they can be expressed in the form 
\begin{align*}
	v_j(x,z_0)=p(x,z_0)\mathrm{e}^{\mathrm{i}\alpha_j(x,z_0)},\quad
	v_j(x,z)=q(x,z)\mathrm{e}^{\mathrm{i}\beta_j(x,z)},
\end{align*}
where $p(x, z_0)$ and $q(x, z)$ denote the magnitudes of $v_j(x, z_0)$ and $v_j(x, z)$, and $\alpha_j(x, z_0)$ and $\beta_j(x, z)$ are the arguments of $v_j(x, z_0)$ and $v_j(x, z)$, respectively. 

We claim that $q(x,z)\not\equiv0$ for all $(x,z)\in\Xi\times\Lambda$. Otherwise, $v_j(x,z)\equiv0$ for all $(x,z)\in\Xi\times\Lambda$. For $z^*\in\Lambda$, there exists a subdomain $\Lambda_0\subset\Lambda$, where $\Lambda_0\in\mathcal{C}^3$, such that $z^*\in\partial\Lambda_0$, and the total field $v_j$ satisfies
\begin{align*}
v_j(x,z^*)\equiv0,\quad 
	\frac{\partial v_j(x,z^*)}{\partial n(x)}=0, \quad \forall\,x\in\partial\Xi\cup\Gamma_j.
\end{align*}
By using the uniqueness of the solution to the direct scattering problem, we have
\begin{align*}
	v_j(x,z^*)=0, \quad \forall\,x\in\mathbb{R}^2\setminus\overline{D_1\cup D_2\cup\{z^*\}},
\end{align*}
which implies 
\begin{align*}
	G(x,z^*)\big|_{\partial\Lambda_0}=-v^s_j(x,z^*)\big|_{\partial\Lambda_0},\quad \frac{\partial G(x,z^*)}{\partial n(x)}\bigg|_{\partial\Lambda_0}=-\frac{\partial v^s_j(x,z^*)}{\partial n(x)}\bigg|_{\partial\Lambda_0}.
\end{align*}

Noting that the scattered field $v^s_j$ is infinitely smooth in the vicinity of $z^*$ and $\partial\Lambda_0\in\mathcal{C}^3$, we obtain 
\begin{align*}
	\Big(v^s_j(\cdot,z^*)\big|_{\partial\Lambda_0},\,  \partial_n v^s_j(\cdot,z^*)\big|_{\partial\Lambda_0}\Big)\in H^{5/2}(\partial\Lambda_0)\times H^{3/2}(\partial\Lambda_0).
\end{align*}
This regularity also applies to  $(G(\cdot,z^*)|_{\partial\Lambda_0},\, \partial_n G(\cdot,z^*)|_{\partial\Lambda_0})$. Therefore, the function $G(\cdot,z^*)$ belongs to $H^3$ in the vicinity of $z*$, which contradicts the fact that $G(\cdot,z^*)\in H^2_{loc}(\mathbb{R}^2)$ (cf. \cite[Lemma 2.2]{BR2020}). By the continuity and reciprocity of $q(x,z)$, there exist subdomains $\widetilde{\Xi}\subset\Xi$ and $\widetilde{\Lambda}_0\subset\Lambda$ such that $q(x,z)\not=0$ for all $(x,z)\in\widetilde{\Xi}\times\widetilde{\Lambda}_0$. Similarly, there exists a subdomain $\widetilde{\Xi}_0\subset\widetilde{\Xi}\subset\Xi$ such that $p(x,z_0)\not=0$ and $q(x,z)\not=0$ for all $(x,z)\in\widetilde{\Xi}_0\times\widetilde{\Lambda}_0$.

According to \eqref{Reequal}, we deduce
\[
\cos[\alpha_1(x,z_0)-\beta_1(x,z)]=\cos[\alpha_2(x,z_0)-\beta_2(x,z)], \quad \forall\,(x,z)\in\widetilde{\Xi}_0\times\widetilde{\Lambda}_0, 
\]
implying that we may either have
\begin{equation}
	\label{case1}
	\zeta(x):=\alpha_1(x,z_0)-\alpha_2(x,z_0)-2m\pi=\beta_1(x,z)-\beta_2(x,z), \quad \forall\,(x,z)\in\widetilde{\Xi}_0\times\widetilde{\Lambda}_0, 
\end{equation}
or 
\begin{equation}
	\label{case2}
	\eta(x):=\alpha_1(x,z_0)+\alpha_2(x,z_0)-2m\pi=\beta_1(x,z)+\beta_2(x,z), \quad \forall\,(x,z)\in\widetilde{\Xi}_0\times\widetilde{\Lambda}_0.
\end{equation}

For case \eqref{case1}, considering
\[
v_1(x,z)=q(x,z)\mathrm{e}^{\mathrm{i}\beta_1(x,z)}=q(x,z)\mathrm{e}^{\mathrm{i}\beta_2(x,z)+\mathrm{i}\zeta(x)}=\mathrm{e}^{\mathrm{i}\zeta(x)}v_2(x,z)
\]
and employing the reciprocity relation \eqref{symmv}, we obtain 
\[
v_1(z,x)=\mathrm{e}^{\mathrm{i}\zeta(x)}v_2(z,x), \quad \forall\,(z,x)\in\widetilde{\Lambda}_0\times\widetilde{\Xi}_0.
\]
Using the continuity leads to 
\[
v_1(z,x)=\mathrm{e}^{\mathrm{i}\zeta(x)}v_2(z,x),\quad\partial_{n(z)}v_1(z,x)=\mathrm{e}^{\mathrm{i}\zeta(x)}\partial_{n(z)}v_2(z,x), \quad \forall\,(z,x)\in\partial\widetilde{\Lambda}_0\times\widetilde{\Xi}_0.
\]
Additionally, utilizing the well-posedness of the direct scattering problem, we affirm that
\[
v_1(z,x)=\mathrm{e}^{\mathrm{i}\zeta(x)}v_2(z,x), \quad \forall\,z\in\mathbb{R}^2\setminus\overline{D_1\cup D_2\cup\{x\}}, \quad \forall\,x\in\widetilde{\Xi}_0.
\] 
In other words,
\[
v_1^s(z,x)+G(z,x)=\mathrm{e}^{\mathrm{i}\zeta(x)}\big(v_2^s(z,x)+G(z,x)\big),  \quad \forall\,z\in\mathbb{R}^2\setminus\overline{D_1\cup D_2\cup\{x\}}, \quad \forall\,x\in\widetilde{\Xi}_0.
\]

We claim that $\mathrm{e}^{\mathrm{i}\zeta(x)}\equiv1$ for all $x\in\widetilde{\Xi}_0$. Otherwise, there exits $ x^*\in\widetilde{\Xi}_0$ such that $\mathrm{e}^{\mathrm{i}\zeta(x^*)}-1\not=0$. Then, there exists a domain $\Xi_0\subset\widetilde{\Xi}$, where $\Xi_0\in\mathcal{C}^3$, such that  $x^*\in\partial\Xi_0$ and
\begin{align*}
G(z,x^*)=\frac{v_1^s(z,x^*)-\mathrm{e}^{\mathrm{i}\zeta(x^*)}v_2^s(z,x^*)}{\mathrm{e}^{\mathrm{i}\zeta(x^*)}-1}, \quad \forall\,z\in\mathbb{R}^2\setminus\overline{D_1\cup D_2\cup\{x^*\}}.
\end{align*}
Noting that the scattered field $v^s_j$ is infinitely smooth in the vicinity of $x^*$ and $\partial\Xi_0\in\mathcal{C}^3$, the above identity implies that $\big(G(\cdot,x^*)\big|_{\partial\Xi_0},\,  \partial_n G(\cdot,x^*)\big|_{\partial\Xi_0}\big)\in H^{5/2}(\partial\Xi_0)\times H^{3/2}(\partial\Xi_0)$. Consequently, the function $G(\cdot,x^*)$ belongs to $H^3$ in the vicinity of $x^*$, which leads to a contradiction. Therefore, we deduce
\[
v^s_1(z,x)=v^s_2(z,x), \quad \forall\,z\in\mathbb{R}^2\setminus\overline{D_1\cup D_2\cup\{x\}}, \quad \forall\,x\in\widetilde{\Xi}_0, 
\] 
which shows that the corresponding far-field patterns coincide, i.e., $v^\infty_{{\rm H},1}(\hat{z},x)=v^\infty_{{\rm H},2}(\hat{z},x)$ for all $(\hat{z},x)\in\Omega\times\widetilde{\Xi}_0$. From the mixed reciprocity relation \eqref{mixrel}, we obtain $u^s_1(x,-\hat{z})=u^s_2(x,-\hat{z})$ for all $(x,\hat{z})\in\widetilde{\Xi}_0\times\Omega$.
Subsequently, considering continuity and the well-posedness of the direct scattering problem, we derive
\[
u^\infty_1(\hat{x},d)=u^\infty_2(\hat{x},d), \quad \forall\,(\hat{x},d)\in\Omega.
\]
Hence, through the utilization of Theorem \ref{uniqueness1}, we conclude that $D_1=D_2$.

The proof is concluded by excluding case \eqref{case2}. Suppose that \eqref{case2} is valid, then following the similar arguments, it can be established that
\[
v_1(z,x)=\mathrm{e}^{\mathrm{i}\eta(x)}\overline{v_2(z,x)}, \quad \forall\,z\in\mathbb{R}^2\setminus\overline{D_1\cup D_2\cup\{x\}}, \quad \forall\,x\in\widetilde{\Xi}_0.
\]
We claim that $\mathrm{e}^{\mathrm{i}\eta(x)}\equiv1$ for all $x\in\widetilde{\Xi}_0$. Otherwise, there exits $x^*\in\widetilde{\Xi}_0$ such that $\mathrm{e}^{\mathrm{i}\eta(x^*)}-1\not=0$. Then, we have
\begin{align}\label{Gcomplex}
G(z,x^*)-\mathrm{e}^{\mathrm{i}\eta(x^*)}\overline{G(z,x^*)}=\mathrm{e}^{\mathrm{i}\eta(x^*)}\overline{v^s_2(z,x^*)}-v^s_1(z,x^*).
\end{align}
The Green function $G(z,x^*)$ can be decomposed as
\[
G(z,x^*)=G_0(z,x^*)+T(z,x^*), \quad z\not=x^*,
\]
where $G_0(z,x^*)=-|z-x^*|^2\log|z-x^*|/(8\pi)\in H^2_{loc}(\mathbb{R}^2)$ is the fundamental solution of the bi-Laplacian operator $\Delta^2$, and $T(z,x^*)$ is an infinitely smooth function. Consequently, the identity \eqref{Gcomplex} can be reformulated as
\begin{align}\label{G0complex}
	(1-\mathrm{e}^{\mathrm{i}\eta(x^*)})G_0(z,x^*)=\mathrm{e}^{\mathrm{i}\eta(x^*)}\overline{v^s_2(z,x^*)+T(z,x^*)}-\big(v^s_1(z,x^*)+T(z,x^*)\big)
\end{align}
for $z\in\mathbb{R}^2\setminus\overline{D_1\cup D_2\cup\{x^*\}}$. Again, through a similar discussion, we have from \eqref{G0complex} that the function $G_0(\cdot,x^*)$ belongs to $H^3$ in the vicinity of $x^*$, which is a contradiction. Hence, we obtain 
\[
v_1(z,x)=\overline{v_2(z,x)}, \quad \forall\,z\in\mathbb{R}^2\setminus\overline{D_1\cup D_2\cup\{x\}}, \quad \forall\,x\in\widetilde{\Xi}_0.
\]

From \eqref{vinftytot} and Lemma \ref{Rellich1}, it is evident that $v^\infty_{{\rm tot},1}(\hat{z},x)\not\equiv0$. By using continuity, there exist $\Omega_0\subset\Omega$ and $\widetilde{\Xi}^0_0\subset\widetilde{\Xi}_0$ such that $v^\infty_{{\rm tot},1}(\hat{z},x)\not=0$ for all $(\hat{z},x)\in\Omega_0\times\widetilde{\Xi}^0_0$. For $\tilde{z}\in\Omega_0$, denote $z=\rho\tilde{z}$. From the definition of the far-field pattern \eqref{farf}, it follows that
\begin{align*}
	\lim_{\rho\to\infty}\rho^{1/2}\mathrm{e}^{-\mathrm{i}\kappa\rho}v_1(\rho\tilde{z},x)=v^\infty_{{\rm tot},1}(\tilde{z},x), \quad \lim_{\rho\to\infty}\rho^{1/2}\mathrm{e}^{-\mathrm{i}\kappa\rho}\overline{v_2(\rho\tilde{z},x)}=\overline{v^\infty_{{\rm tot},2}(\tilde{z},x)}.
\end{align*}
Additionally, observing that $v^\infty_{{\rm tot},1}(\hat{z},x)\neq 0$ and $v_1(z,x)=\overline{v_2(z,x)}$, we deduce
\[
\lim_{\rho\to\infty}\mathrm{e}^{2\mathrm{i}\kappa\rho}=\frac{\overline{v^\infty_{{\rm tot},2}(\tilde{z},x)}}{v^\infty_{{\rm tot},1}(\tilde{z},x)},
\]
which constitutes a contradiction since the limit of the left-hand side does not exist. Therefore, case \eqref{case2} is not valid and the proof is completed. 
\end{proof}

\section{Conclusion}\label{S5}

In this paper, we have studied the inverse cavity scattering problem concerning the two-dimensional biharmonic wave equation. Initially, employing the decomposition of the biharmonic equation, we transform the original scattering problem into an equivalent coupled boundary value problem. Subsequently, we demonstrate the Green's representation of the solution for the original boundary value problem. This leads us to derive two types of mixed reciprocity relations, connecting the scattered field produced by plane waves to the far-field pattern generated by various point sources. Moreover, we investigate the symmetry relation of the scattered field resulting from different point sources and establish a one-to-one correspondence between the Helmholtz component of the biharmonic wave and the far-field pattern. We proceed to establish the uniqueness result using the far-field pattern with a fixed wavenumber. Additionally, building upon this uniqueness result, we establish uniqueness of the inverse cavity scattering problem from the phaseless near-field data. 

It is worth mentioning that the results presented in this work can be readily extended to the three-dimensional problem with straightforward modifications. However, we opted for the two-dimensional setting to demonstrate the results due to the motivation stemming from thin plate elasticity, where the physical validation principally occurs within a two-dimensional framework.

This work primarily concentrates on the clamped boundary condition. Our aim is to expand upon these findings to include the diverse range of boundary conditions frequently encountered in elastic thin plates. Moreover, our current research involves exploring uniqueness using phaseless far-field patterns and employing numerical methods to address the inverse cavity scattering problem. The progress made in these investigations will be presented in a future publication.

\end{document}